\documentclass[a4paper, leqno, 12pt]{article}
\usepackage{amsmath,amsthm}
\usepackage{amsfonts}
\usepackage{amssymb,latexsym,apptools}
\usepackage{enumerate}
\usepackage{color}
\usepackage{array}
\usepackage[all]{xy}
\usepackage{apptools}

\usepackage{amsrefs}

\usepackage[left=3cm,right=3cm,top=3.0cm,bottom=3.0cm]{geometry}

\newtheorem{thm}{Theorem}[section]
\newtheorem{cor}[thm]{Corollary}
\newtheorem{lm}[thm]{Lemma}
\newtheorem{prop}[thm]{Proposition}

\theoremstyle{definition}
\newtheorem*{xrem}{Remark}

\numberwithin{equation}{section}

\def\NN{{\mathbb{N}}}
\def\ZZ{{\mathbb{Z}}}

\def\CA{{\cal A}}
\def\CC{{\cal C}}
\def\CB{{\cal B}}
\def\CD{{\cal D}}
\def\CE{{\cal E}}
\def\CF{{\cal F}}
\def\CG{{\cal G}}

\def\CK{{\cal K}}

\def\CR{{\cal R}}
\def\CS{{\cal S}}

\def\epv {{ $\mbox{}$\hfill \qedsymbol }}

\def\Im{\textnormal{Im}}

\def\wt{\widetilde}
\def\ra{\rightarrow}

\def\mod{\mbox{{\rm mod}}}

\def\ind{\mbox{{\rm ind}}}

\def\Mod{\mbox{{\rm Mod}}}
\def\Ext{\mbox{{\rm Ext}}}
\def\Coker{\mbox{{\rm Coker}}}  
\def\supp{\mbox{\rm supp}}

  \def\ind{{\rm ind}}

\DeclareMathOperator{\ob}{ob}

\let\mod=\undefined
\DeclareMathOperator{\KG}{KG}
\DeclareMathOperator{\Ker}{Ker}
\DeclareMathOperator{\op}{op}
\DeclareMathOperator{\add}{add}
\DeclareMathOperator{\End}{End}
\DeclareMathOperator{\Hom}{Hom}
\DeclareMathOperator{\mod}{mod}
\DeclareMathOperator{\KGdim}{KG-dim}

\begin{document}

\baselineskip=17pt


\title{On Krull-Gabriel dimension of cluster repetitive categories and cluster-tilted algebras}

\author{Alicja Jaworska-Pastuszak${}^{*}$ and Grzegorz Pastuszak${}^{*}$\\ with an Appendix by Grzegorz Bobi{\'n}ski${}^{*}$}

\date{}

\maketitle

\renewcommand{\thefootnote}{}
\footnote{${}^{*}$Faculty of Mathematics and Computer Science, Nicolaus Copernicus University, Chopina 12/18, 87-100 Toru\'n, Poland, e-mails:
gregbob@mat.umk.pl (G.~Bobi{\'n}ski), jaworska@mat.umk.pl (A.~Jaworska-Pastuszak), past@mat.umk.pl (G.~Pastuszak).}

\renewcommand{\thefootnote}{\arabic{footnote}}
\setcounter{footnote}{0}

\begin{abstract} Assume that $K$ is an algebraically closed field and denote by $\KG(R)$ the Krull-Gabriel dimension of $R$, where $R$ is a locally
bounded $K$-category (or a bound quiver $K$-algebra). Assume that $C$ is a tilted $K$-algebra and $\widehat{C},\check{C},\wt{C}$ are the associated
repetitive category, cluster repetitive category and cluster-tilted algebra, respectively. Our first result states that
$\KG(\wt{C})=\KG(\check{C})\leq\KG(\widehat{C})$. Since the Krull-Gabriel dimensions of tame locally support-finite repetitive categories are
known, we further conclude that $\KG(\wt{C})=\KG(\check{C})=\KG(\widehat{C})\in\{0,2,\infty\}$. Finally, in the Appendix Grzegorz Bobi{\'n}ski
presents a different way of determining the Krull-Gabriel dimension of the cluster-tilted algebras, by applying results of Geigle.
\end{abstract}

\section{Introduction}

Assume that $K$ is an algebraically closed field and $R$ is a locally bounded $K$-category. Recall that $R$ is isomorphic to a bound quiver
$K$-category associated with some locally finite bound quiver \cites{BoGa,Ga}. Hence we can identify finite locally bounded $K$-categories with bound
quiver $K$-algebras, see \cite{AsSiSk}*{I-III}. We denote by $\mod(K)$ the category of all finite dimensional $K$-vector spaces,
 by $\mod(R)$ the category of all finitely generated right $R$-modules and by $\ind(R)$ the category
of all finitely generated indecomposable right $R$-modules. Let $\CF(R)$ be the category of all finitely presented contravariant $K$-linear
functors from $\mod(R)$ to $\mod(K)$, see \cite{Pr2} for details on functor categories. A natural approach to study this abelian category is via the
associated \textit{Krull-Gabriel filtration} \cite{Po}
$$\CF(R)_{-1}\subseteq\CF(R)_{0}\subseteq\CF(R)_{1}\subseteq\hdots\subseteq\CF(R)_{\alpha}\subseteq\CF(R)_{\alpha+1}\subseteq\hdots$$ of $\CF(R)$ by
Serre subcategories, defined recursively as follows:
\begin{enumerate}[\rm(1)]
	\item $\CF(R)_{-1}=0$ and $\CF(R)_{\alpha+1}$ is the Serre subcategory of $\CF(R)$ formed by all functors having finite length in the quotient
category $\CF(R)\slash\CF(R)_{\alpha}$, where $\alpha$ is an ordinal number or $\alpha=-1$,
	\item $\CF(R)_{\beta}=\bigcup_{\alpha<\beta}\CF(R)_{\alpha}$, for any limit ordinal $\beta$.
\end{enumerate}
The \textit{Krull-Gabriel dimension} $\KG(R)$ of $R$ \cite{Ge2} is the smallest ordinal number $\alpha$ such that $\CF(R)_{\alpha}=\CF(R)$, if such
a number exists, and $\KG(R)=\infty$ otherwise. The Krull-Gabriel dimension of $R$ is \textit{finite} if $\KG(R)=n$, for some $n\in\NN$. The
Krull-Gabriel dimension of $R$ is \textit{undefined} if $\KG(R)=\infty$. We note that the Krull-Gabriel dimension is defined for any small abelian
category similarly as above.

There are several motivations to study the Krull-Gabriel dimension of a locally bounded $K$-category or a bound quiver $K$-algebra, see \cite{Sch3}
for details. Our motivation comes from the conjecture of Prest \cite{Pr2} stating that a finite dimensional algebra $A$ is of domestic
representation type if and only if the Krull-Gabriel dimension $\KG(A)$ of $A$ is finite. We refer to \cite{SiSk3}*{XIX} for the definitions of
finite, tame and wild representation type as well as stratification of tame representation type into domestic, polynomial and non-polynomial growth
(introduced in \cite{SkBC}).

The known results support the conjecture of Prest, see the introduction of \cite{P4} for a comprehensive and up-to-date list of these results. Let us
recall that the first and fundamental fact in this direction, due to Auslander \cite{Au}*{Corollary 3.14}, states that an algebra $A$ is of finite
representation type if and only if $\KG(A)=0$. The present paper is in part devoted to determination of the Krull-Gabriel dimension of the
\textit{cluster-tilted algebras}. In particular, we confirm the conjecture of Prest for this class of algebras. Recall that the cluster-tilted
algebras play a prominent role in the theory of \textit{cluster algebras}.

The cluster algebras were introduced by Fomin and Zelevinsky in the seminal paper \cite{FoZe1} in order to create a combinatorial framework for the
study of canonical bases in quantum groups and for the study of total positivity for algebraic groups. Since then the cluster algebras have become a
separate field of study with many connections to other subjects, in particular to the representation theory of finite dimensional algebras. We refer
the reader to \cite{Ke} for a comprehensive survey on the theory of cluster algebras and related topics.

Cluster algebras are linked to the representation theory via tilting theory in \textit{cluster categories} which was introduced by Buan,  Marsh,
Reineke,  Reiten and Todorov in another seminal paper \cite{BMRRT}. If $H$ is a hereditary algebra and $\mathcal{D}^{b}(H)$ is the derived category
of bounded complexes over $\mod(H)$, then the \textit{cluster category} $\CC_{H}$ is an orbit category of $\mathcal{D}^{b}(H)$ under the action of
the functor $\tau^{-1}[1]$. Here $\tau$ denotes the Auslander-Reiten translation in $\mathcal{D}^{b}(H)$ and $[1]$ the shift functor. Recall that
Keller shows in \cite{Ke1} a deep result that the cluster category $\CC_{H}$ is triangulated.

An object $T$ in $\CC_{H}$ is a \textit{cluster-tilting object} provided $T$ has no self-extensions in $\CC_{H}$ and the number of isomorphism
classes  of indecomposable direct summands of $T$ equals the number of simple modules in $\mod(H)$. To each hereditary algebra $H$ one can associate
a cluster algebra in such a way that the cluster variables correspond to the indecomposable direct summands of cluster-tilting objects in $\CC_{H}$ and the
clusters to the cluster-tilting objects themselves, see \cite{BMRRT} for details. In this way cluster categories \textit{categorify} cluster algebras.

\textit{Cluster-tilted algebras} were introduced in \cite{BMR} as (the opposite algebras of) the endomorphisms algebras of cluster-tilting objects.
This is analogous to the classical definition of tilted algebras by Happel and Ringel from \cite{HR}. Cluster-tilted algebras attracted much
attention and are still an active area of research in cluster theory, see \cite{As} for a convenient survey on the topic. Interestingly, it turns out
that these algebras can be defined without any reference to cluster categories. Indeed, it is proved in \cite{ABS2} that any cluster-tilted algebra
is of the form of the trivial extension $\wt{C}:=C\ltimes\Ext^{2}(DC,C)$, where $C$ is a tilted algebra and $D$ denotes the standard $K$-duality.
This fact allows to show in \cite{ABS} that a cluster-tilted algebra $\wt{C}$ is an orbit algebra of a \textit{cluster repetitive category}
$\check{C}$. More specifically, there is a \textit{Galois covering} \cites{BoGa,Ga} $\check{C}\ra\wt{C}$ with a covering group $\ZZ$, see Section 3
for details. Note that this point of view is similar to the description of standard self-injective algebras as orbit algebras of \textit{repetitive
categories}, see \cite{Sk4}.

Recently the second author showed in \cites{P4,P6} a general result that if $R$ is a \textit{locally support-finite} \cite{DoSk} locally bounded
$K$-category and $G$ a torsion-free admissible group of $K$-linear automorphisms of $R$, then $\KG(R)=\KG(R\slash G)$ where $R\slash G$ denotes the
orbit category \cite{BoGa}. In other words, the induced Galois covering $R\ra R\slash G$ preserves Krull-Gabriel dimension. This theorem is further
applied in \cite{P4}*{Theorem 7.3} and \cite{P4}*{Theorem 8.1} to determine Krull-Gabriel dimensions of tame locally-support finite repetitive
categories and standard self-injective algebras, respectively.

In the present paper we use the above facts to determine the Krull-Gabriel dimension of cluster repetitive categories and cluster-tilted algebras,
see Theorem 3.6 for a precise statement of the main result. In particular, we show that if $C$ is a tilted $K$-algebra and
$\widehat{C},\check{C},\wt{C}$ are the associated repetitive category, cluster repetitive category and cluster-tilted algebra, respectively, then we
have $$\KG(\wt{C})=\KG(\check{C})=\KG(\widehat{C})\in\{0,2,\infty\}.$$ We also confirm the conjecture of Prest for the class of cluster-tilted
algebras.

The paper is organized as follows. In the remaining part of Section 1 we fix the notation and terminology used in the paper.

Section 2 is devoted to the introduction of \textit{admissible functors} between module categories and their examples. A crucial property of an
admissible  functor $\varphi\colon \mod(A)\ra\mod(B)$ is that its existence implies that $\KG(B)\leq\KG(A)$, see Proposition 2.1. In Theorem 2.2 we
recall some examples of admissible functors from \cite{P6}, related with certain Galois coverings. These admissible functors are used in proofs of
the main results of \cites{P4,P6} (see Theorem 2.3) which are directly applied in the present paper. Then we prove Theorem 2.4 which yields a new class of
admissible functors playing a prominent role in the proofs our main results, given in Section 3. Later on we introduce \textit{hs-finite} classes of modules and
show in Lemma 2.5 their significance in the context of Theorem 2.4.

The most import technical fact of Section 3 is Proposition 3.3. In this proposition we show that some particular class of modules over the repetitive
category of a tilted algebra, studied in \cite{ABS}, is hs-finite. This property enables us to make use of facts from Section 2 in the proofs of the
main results of the paper which are Theorem~3.4, Theorem~3.6 and Corollary~3.7. The crucial part of Theorem~3.4 states that
$\KG(\wt{C})=\KG(\check{C})\leq\KG(\widehat{C})$ where $C$ is a tilted algebra and $\widehat{C},\check{C},\wt{C}$ are the associated repetitive
category, cluster repetitive category and cluster-tilted algebra, respectively. Theorem~3.6 strengthens Theorem~3.4, showing in particular that
$\KG(\wt{C})=\KG(\check{C})=\KG(\widehat{C})\in\{0,2,\infty\}$. In Corollary 3.7 we conclude that if the base field is countable, then $\wt{C}$
possesses a super-decomposable pure-injective module (see \cites{Pr2,KaPa3}) if and only if $\wt{C}$ is of non-domestic type, confirming (another)
conjecture of Prest for the class of cluster-tilted algebras.

The final section of the paper is the Appendix written by Grzegorz Bobi{\'n}ski. The author gives an alternative proof of Theorem~\ref{3.6}, which is based on a result of Geigle given in \cite{Ge2}*{Corollary 2.9}.

Throughout the paper, we use the following notation and terminology. Fix an algebraically closed field $K$ and assume that $R$ is a locally bounded $K$-category. Recall that in this case $R$ is isomorphic with a \textit{bound quiver $K$-category} of a locally finite quiver, see \cite{BoGa} or \cite[Section 2]{P4} for more straightforward presentation.

Assume that $\CC$ is a full subcategory of $R$. Then $\CC$ is \textit{convex} if and only if for any $n\geq 1$ and objects $x,z_{1},\hdots,z_{n},y$ of $R$ the following condition is satisfied: if $x,y$ are objects of $\CC$ and the $K$-vector spaces of morphisms $R(x,z_{1})$, $R(z_{1},z_{2}),\hdots,R(z_{n-1},z_{n})$, $R(z_{n},y)$ are nonzero, then $z_{1},\hdots,z_{n}$ are objects of $\CC$.

Assume that $R$ is a bound quiver $K$-category of a bound quiver $(Q,I)$ and $\CC$ is a full subcategory of $R$. In this setting, $\CC$ is convex if and only if $\CC$ is a bound quiver $K$-category of a bound quiver $(Q',I')$ such that $Q'$ is some convex subquiver of $Q$. Recall that a full subquiver $Q'$ of $Q$ is \textit{convex} if and only if any path in $Q$ whose source and target belong to $Q'$ is entirely contained in $Q'$. We refer to Section 2 of \cite{P4} for more details.

Assume that $\CC$ is a full subcategory of $R$. A \emph{convex hull} of $\CC$ is a full subcategory of $R$ whose set of objects is an intersection of sets of objects of all convex subcategories of $R$ containing $\CC$. Thus a convex hull of $\CC$ is the smallest convex subcategory of $R$ containing $\CC$. We say that $R$ is \textit{intervally-finite} if and only if a convex hull of any finite full subcategory of $R$ is again finite. This notion is only implicitly contained in \cite[2.1]{BoGa}. The terminology was introduced later.

Assume that $R$ is a locally bounded $K$-category. A \textit{right $R$-module} is a $K$-linear contravariant functor of the form $M \colon R\ra\Mod(K)$ where $\Mod(K)$ denotes the category of all $K$-vector spaces. The category of all such modules is denoted by $\Mod(R)$. A module $M\in\Mod(R)$ is \emph{finite dimensional} if and only if $\dim M=\sum_{x\in\ob(R)}\dim_{K} M(x)<\infty$. We denote by $\mod(R)$ the full subcategory of $\Mod(R)$ formed by finite dimensional modules. Moreover, $\ind(R)$ is the full subcategory of $\mod(R)$ whose objects are indecomposable modules.

If $M\in\mod(R)$, then the set $\supp(M)=\{x\in R\mid M(x)\neq 0\}$ is the \textit{support of the module} $M$. If $\CC$ is a full subcategory of $\mod(R)$, then the union of all supports of modules belonging to $\CC$ is called the \textit{support of $\CC$} and denoted by $\supp(\CC)$.

The category $R$ is \textit{locally support-finite} \cite{DoSk} if and only if for any $x\in R$ the union of the sets $\supp(M)$, where $M\in\ind(R)$ and $M(x)\neq 0$, is finite.

Assume that $R,A$ are locally bounded $K$-categories, $F:R\ra A$ is a $K$-linear functor and $G$ a group of $K$-linear automorphisms of $R$ acting freely on the objects of $R$ (i.e. $gx=x$ if and only if $g=1$, for any $g\in G$ and $x\in\ob(R)$). Then $F:R\ra A$ is a \textit{Galois covering} \cite{BoGa} if and only if
\begin{enumerate}[\rm(1)]
	\item the functor $F:R\ra A$ induces isomorphisms $$\bigoplus_{g\in G}R(gx,y)\cong A(F(x),F(y))\cong\bigoplus_{g\in G}R(x,gy)$$ of vector spaces, for any $x,y\in\ob(R)$,
	\item the functor $F:R\ra A$ is surjective on objects,
	\item $Fg=F$, for any $g\in G$,
	\item for any $x,y\in\ob(R)$ such that $F(x)=F(y)$ there is $g\in G$ such that $gx=y$.
\end{enumerate} It is well known that a functor $F:R\ra A$ satisfies the above conditions if and only if $F$ induces an isomorphism $A\cong R\slash G$ where $R\slash G$ is the \textit{orbit category}, see \cite{BoGa}. We refer to \cite{Ga} for a general definition of a covering functor.

Assume that $F:R\ra A\cong R\slash G$ is a Galois covering. The \textit{pull-up} functor $F_{\bullet}:\Mod(A)\ra\Mod(R)$ associated with $F$ is the functor $(-)\circ F^{\op}$. The pull-up functor has the left adjoint $F_{\lambda}:\Mod(R)\ra\Mod(A)$ and the right adjoint $F_{\rho}:\Mod(R)\ra\Mod(A)$ which are called the \textit{push-down functors}. We refer to Section 2 of \cite{P4} for concrete description of these functors. Here we only mention that push-down functors restrict to the categories of finite dimensional modules and these restrictions coincide, that is, $F_{\lambda}(\mod(R))\subseteq\mod(A)$, $F_{\rho}(\mod(R))\subseteq\mod(A)$ and $F_{\lambda}|_{\mod(R)}=F_{\rho}|_{\mod(R)}$. The restriction $F_{\lambda}|_{\mod(R)}$ is also denoted as $F_{\lambda}$. In the paper we consider only the push-down $F_{\lambda}:\mod(R)\ra\mod(A)$.

Furthermore, $\CG(R)$ denotes the category of all contravariant additive $K$-linear functors $\mod(R)\ra\mod(K)$. If $M\in\mod(R)$, then
${}_{R}(-,M)=\Hom{}_{R}(-,M)$  denotes the \textit{contravariant hom-functor}. Recall that any homomorphism of modules $f\in{}_{R}(M,N)$ induces a
homomorphism of functors ${}_{R}(-,f)\colon {}_{R}(-,M)\ra{}_{R}(-,N)$ such that the map ${}_{R}(X,f)\colon {}_{R}(X,M)\ra{}_{R}(X,N)$ is defined by
${}_{R}(X,f)(g)=fg$, for any $g\in{}_{R}(X,M)$. The Yoneda lemma implies that the function $f\mapsto{}_{R}(-,f)$ defines an isomorphism
${}_{R}(M,N)\ra{}_{\CG(R)}({}_{R}(-,M),{}_{R}(-,N))$ of vector spaces and this yields $M\cong N$ if and only if ${}_{R}(-,M)\cong{}_{R}(-,N)$.

Assume that $F\in\CG(R)$. The functor $F$ is \textit{finitely presented} if and only if there exists an exact sequence of functors
${}_{R}(-,M)\xrightarrow{{}_{R}(-,f)}{}_{R}(-,N)\ra F\ra 0$, for some $M,N\in\mod(R)$ and $R$-module homomorphism $f\colon M\ra N$. This means that
$F\cong\Coker{}_{R}(-,f)$ which yields $F(X)$ is isomorphic to the cokernel of the map ${}_{R}(X,f)\colon {}_{R}(X,M)\ra{}_{R}(X,N)$. We denote by
$\CF(R)$ the full subcategory of $\CG(R)$ formed by finitely presented functors. Obviously ${}_{R}(-,M)\in\CF(R)$ for any $M\in\mod(R)$. Moreover,
the functor ${}_{R}(-,M)$ is a projective object of the category $\CF(R)$ and any projective object of $\CF(R)$ is a hom-functor, see
\cite{AsSiSk}*{IV.6}. If $F\in\CG(R)$, then $\supp(F)=\{X\in\mod(R)\mid F(X)\neq 0\}$ is the \textit{support of $F$}.

We refer the reader to \cite{AsSiSk} for the background of the representation theory of finite dimensional algebras over algebraically closed fields.

\section{Admissible functors and Krull-Gabriel dimension}

In this section we introduce \textit{admissible functors} and relate them with Krull-Gabriel dimension. We give examples of such functors  in
Theorems \ref{2.3} and \ref{2.6}. These results are applied in Section 3.

Assume that $\varphi\colon \mod(A)\ra\mod(B)$ is a $K$-linear additive covariant functor. We define $\Lambda_{\varphi}\colon \CF(B)\ra\CG(A)$ as the composition $(-)\circ\varphi$. Observe that if $U\in\CF(B)$ and ${}_{B}(-,X)\xrightarrow{_{B}(-,f)}{}_{B}(-,Y)\ra U\ra 0$ is exact, then we get the exact sequence $${}_{B}(\varphi(-),X)\xrightarrow{_{B}(\varphi(-),f)}{}_{B}(\varphi(-),Y)\ra U\varphi\ra 0.$$
We say that $\varphi \colon \mod(A)\ra\mod(B)$ is \textit{admissible} if and only if $\varphi$ is dense\footnote{This means that for any module $X\in\mod(B)$ there exists a module $M\in\mod(A)$ such that $\varphi(M)\cong X$. Although the property is often referred as \emph{essential surjectivity}, we stick to above terminology since it is consistent with our previous work, especially with \cite{P4}.} and $\Im(\Lambda_{\varphi})\subseteq\CF(A)$,
that is, $U\varphi$ is finitely presented, for any $U\in\CF(A)$.

The following fact shows that admissible functors are useful in the study of Krull-Gabriel dimension.

\begin{prop} \label{2.2}
Assume that $\varphi\colon \mod(A)\ra\mod(B)$ is an admissible functor. Then we have $\KG(B)\leq\KG(A)$.
\end{prop}

\begin{proof} The functor $\Lambda_{\varphi}\colon \CF(B)\ra\CF(A)$ is exact being a composition with a $K$-linear additive functor. We show that $\Lambda_{\varphi}$ is also faithful. Indeed, let $f=(f_{N})_{N}:U\ra V$ be a natural transformation of functors $U,V\in\CF(A)$ and assume that $\Lambda_{\varphi}(f)=0$. If $N\in\mod(A)$, then $N\cong\varphi(X)$, for some $X\in\mod(B)$ and thus $f_{N}\cong f_{\varphi(X)}=\Lambda_{\varphi}(f)_{X}=0$. This yields that $f=0$ and hence the functor $\Lambda_{\varphi}$ is faithful. Then we conclude from \cite[Appendix B]{Kr} that $\KG(B)\leq\KG(A)$.
\end{proof}

The following two theorems (Theorem \ref{2.3} and \ref{2.4}) are proved in \cite{P4} and \cite{P6}. The assertion $(3)$ of Theorem \ref{2.3} gives an
interesting example of an  admissible functor studied in \cite{P6}.

Assume that $R$ is a locally bounded $K$-category, $G$ is an admissible group of $K$-linear automorphisms of $R$ and $F\colon R\ra A\cong R\slash G$
the associated Galois covering. Assume that $B$ is a finite convex subcategory of the category $R$. We denote by $\CE_{B}\colon \mod(B)\ra\mod(R)$
the \textit{functor of extension by zeros}. We call $B$ a \textit{fundamental domain} of the category $R$ if and only if for any $M\in\ind(R)$ there
exists $g\in G$ such that $\supp({}^{g}M)\subseteq B$, see \cite{P6} (by $^{g}M$ we denote the induced action $^{g}M=M\circ g^{-1}$ of $G$ on
$\mod(R)$). We recall the following theorem proved in \cite{P6}.

\begin{thm} \label{2.3}
Assume that $R$ is a locally bounded $K$-category and $G$ an admissible torsion-free group of $K$-linear automorphisms of $R$. The following
assertions hold.
\begin{enumerate}[\rm(1)]
	\item If there exists a fundamental domain $B$ of $R$, then $R$ is locally support-finite.
	\item If $R$ is locally support-finite and intervally-finite, then there exists a fundamental domain $B$ of $R$.
	\item If $B$ is a fundamental domain of $R$, then the push-down functor $F_{\lambda}\colon \mod(R)\ra\mod(A)$ is dense and the functor
$F_{\lambda}\CE_{B}\colon \mod(B)\ra\mod(A)$ is admissible. In particular, $\KG(A)\leq\KG(B)$ and thus $\KG(A)\leq\KG(R)$.
\end{enumerate}
\end{thm}

\begin{proof} All assertions are proved in \cite{P6} and \cite{P4}. For convenience we present a short proof of $(1)$. Assume, to the contrary, that $R$ is
not locally support-finite. Then there are indecomposable finite dimensional $R$-modules with arbitrarily large supports, because $R$ is a bound
quiver $K$-category of a locally finite quiver. In particular, if $B$ is any finite subcategory of $R$, then there is $M\in\ind(R)$ such that
$|\supp(M)|>|B|$. Hence there is no $g\in G$ with $\supp({}^{g}M)\subseteq B$ and so $B$ is not a fundamental domain of $R$.
\end{proof}

The following theorem is the main result of the papers \cite{P4} and \cite{P6}, see in particular \cite{P6}*{Theorem 1.5} (Theorem \ref{2.3} is an
important ingredient of its proof). We apply this theorem in the next section.

\begin{thm} \label{2.4}
Assume $R$ is a locally bounded $K$-category, $G$ an admissible torsion-free group of $K$-linear automorphisms of $R$ and $F\colon R\ra A\cong
R\slash G$ the Galois covering. If $B$ is a fundamental domain of $R$, then $\KG(R)=\KG(B)=\KG(A)$.  \qedsymbol
\end{thm}

In the sequel we assume that $A,B$ are arbitrary locally bounded $K$-categories. We aim to present other examples of admissible functors. For that
purpose, we introduce the following definition.

Assume that $\CR\subseteq\mod(A)$ is a class of $A$-modules and $N\in\mod(A)$. A homomorphism $\alpha_{N}\colon M_{N}\ra N$, where $M_{N}\in\CR$, is
a \textit{right $\CR$-approximation} of $N$ if and only if for any $L\in\CR$ and $a\colon L\ra N$ there is $b\colon L\ra M_{N}$ such that
$\alpha_{N}b=a$, that is, the following diagram $$\xymatrix{&N\\ L\ar[ur]^{a}\ar[r]^{b}&M_{N}\ar[u]_{\alpha_{N}}}$$ is commutative. Further, we say
that $\CR$ is \textit{contravariantly finite} if and only if any module $N\in\mod(A)$ has a right $\CR$-approximation.

Observe that if $\CR\subseteq\mod(A)$ is a contravariantly finite class of $A$-modules and $\CS$ is the smallest full subcategory of $\mod(A)$ closed
under isomorphisms and direct summands such that $\CR\subseteq \ob(\CS)$, then $\CS$ is a contravariantly finite subcategory of $\mod(A)$ in the
classical sense of \cite{AuRe}.

Assume that $\varphi:\mod(A)\ra\mod(B)$ is a $K$-linear additive covariant functor and $\CR\subseteq\mod(A)$ is a class of $A$-modules. We denote by
$\Ker(\varphi)$ the \textit{kernel of $\varphi$}, that is, the class of all homomorphisms $f$ in $\mod(A)$ such that $\varphi(f)=0$. We say that a
homomorphism $f:X\ra Y$ in $\mod(A)$ \textit{factorizes through $\CR$} if and only if there is a module $M\in\CR$
 and homomorphisms $g:X\ra M$ and $h:M\ra Y$ in $\mod(A)$ such that $f=hg$.

The following theorem shows important examples of admissible functors.

\begin{thm} \label{2.6}
Assume that $\varphi:\mod(A)\ra\mod(B)$ is a $K$-linear additive covariant functor which is full and dense. Moreover, assume that there is a
contravariantly finite class of modules $\CR_{\varphi}\subseteq\mod(A)$ such that $\Ker(\varphi)$ equals the class of all homomorphisms in $\mod(A)$
which factorize through $\CR_{\varphi}$. Then $\varphi:\mod(A)\ra\mod(B)$ is admissible and thus $\KG(B)\leq\KG(A)$.
\end{thm}

\begin{proof} It is enough to show that a functor ${}_{B}(\varphi(-),Z)\colon \mod(A)\ra\mod(K)$ belongs to $\CF(A)$, for any $Z\in\mod(B)$. Indeed, this
follows from the fact that for any $U\in\CF(B)$ the functor $U\varphi\in\CG(A)$ is a cokernel of a morphism between such functors (see the beginning
of this section) and the category $\CF(A)$ is abelian. Since $\varphi\colon \mod(A)\ra\mod(B)$ is dense, it is sufficient to show that
${}_{B}(\varphi(-),\varphi(N))\in\CF(A)$, for any $N\in\mod(A)$. We fix a module $N\in\mod(A)$.

Observe that $\wt{\varphi}=(\wt{\varphi}_{X})_{X\in\mod(A)}$ where $\wt{\varphi}_{X}\colon {}_{A}(X,N)\ra{}_{B}(\varphi(X),\varphi(N))$ is given by
the formula $\wt{\varphi}_{X}(f)=\varphi(f)$, for any $X\in\mod(A)$ and $f\in{}_{A}(X,N)$, is a natural transformation of functors
${}_{A}(-,N)\ra{}_{B}(\varphi(-),\varphi(N))$. Namely, the fact that the functor $\varphi$ preserves the composition (as any covariant functor)
implies that the following diagram
$$\xymatrix{{}_{A}(X,N)\ar[rr]^{\wt{\varphi}_{X}}\ar[d]^{(-)\circ g}&&
{}_{B}(\varphi(X),\varphi(N))\ar[d]^{(-)\circ\varphi(g)}\\{}_{A}(Y,N)\ar[rr]^{\wt{\varphi}_{Y}}&& {}_{B}(\varphi(Y),\varphi(N))}$$ commutes, for any
homomorphism $g\colon Y\ra X\in\mod(A)$. Since $\varphi$ is full, we get that $\wt{\varphi}\colon {}_{A}(-,N)\ra{}_{B}(\varphi(-),\varphi(N))$ is an
epimorphism of functors.

Assume that $\alpha_{N}\colon M_{N}\ra N$ is a right $\CR_{\varphi}$ approximation of the module $N$, for some $M_{N}\in\CR_{\varphi}$. We show that
$\Im({}_{A}(-,\alpha_{N}))=\Ker(\wt{\varphi})$, equivalently, the sequence
$${}_{A}(-,M_{N})\xrightarrow{_{A}(-,\alpha_{N})}{}_{A}(-,N)\stackrel{\wt{\varphi}}{\longrightarrow}{}_{B}(\varphi(-),\varphi(N))\ra 0$$
is exact. We show that $\Im({}_{A}(X,\alpha_{N}))=\Ker(\wt{\varphi}_{X})$, for any $X\in\mod(A)$. Indeed, if $a\in\Im(X,\alpha_{N})$, then
$a=\alpha_{N}b$, for some $b\in{}_{A}(X,M_{N})$ and thus $a$ factorizes through $M_{N}\in\CR_{\varphi}$. This yields $a\in\Ker(\wt{\varphi}_{X})$ and
hence $\Im({}_{A}(X,\alpha_{N}))\subseteq\Ker(\wt{\varphi}_{X})$. To show the converse inclusion, assume that $f\in\Ker(\wt{\varphi}_{X})$. Then $f$
factorizes through some $L\in\CR_{\varphi}$, so $f=hg$, for some homomorphisms $g\colon X\ra L$ and $h\colon L\ra N$. Since $R_{\varphi}$ is
contravariantly finite, we get that $h=\alpha_{N}j$, for some $j\colon L\ra M_{N}$. This implies $f=hg=\alpha_{N}jg\in\Im({}_{A}(X,\alpha_{N}))$
which shows the second inclusion and proves that the above exact sequence gives a projective presentation of the functor
${}_{B}(\varphi(-),\varphi(N))$.

Summing up, we get ${}_{B}(\varphi(-),\varphi(N))\in\CF(A)$, for any module $N\in\mod(A)$, which yields $U\varphi\in\CF(A)$, for any $U\in\CF(B)$ and
so the functor $\varphi\colon \mod(A)\ra\mod(B)$ is admissible. Consequently, we get $\KG(B)\leq\KG(A)$ by Proposition \ref{2.2}.
\end{proof}

Assume that $\CR\subseteq\mod(A)$ is some class of $A$-modules. If $T\in\CF(A)$, then the class $\supp_{\CR}(T)=\{X\in\CR\mid T(X)\neq 0\}$ is called
the \textit{$\CR$-support} of $T$. We shall call the class $\CR$ \textit{hom-support finite} (in short, \textit{hs-finite}) if and only if the
$\CR$-support of a hom-functor ${}_{A}(-,N)$ is finite, for any $N\in\mod(A)$.

The following lemma is a generalized version of a well-known fact for finite subcategories of $\mod(A)$, see for example \cite{AuSm}*{Proposition
4.2}. We apply the lemma in the next section. Let us denote by $\add(\CR)$ the class of all finite direct sums of modules from the class $\CR$.

\begin{lm} \label{2.7}
Assume that $\CR\subseteq\mod(A)$ is a hs-finite class of $A$-modules. Then the class $\add(\CR)$ is contravariantly finite.
\end{lm}

\begin{proof} Assume that $N\in\mod(A)$. We set $$M_{N}=\bigoplus_{X\in\CR}({}_{A}(X,N)\otimes_{K}X)$$ and define
$\alpha_{N}:M_{N}\ra N$ as $\alpha_{N}(f\otimes x)=f(x)$, for any $f\in {}_{A}(X,N)$ and $x\in X$. Observe that $M_{N}$ is a finite dimensional
$A$-module, because $\CR\subseteq\mod(A)$ is hs-finite and so there is only a finite number of modules $X\in\CR$ such that ${}_{A}(X,N)\neq 0$. We
show that $\alpha_{N}:M_{N}\ra N$ is a right $\add(\CR$)-approximation of $N$. Indeed, for $X\in\CR$ and a homomorphism $a:X\ra N$ define $b:X\ra
M_{N}$ as $b(x)=a\otimes x$, for any $x\in X$. Then we have $(\alpha_{N}b)(x)=\alpha_{N}(b(x))=\alpha_{N}(a\otimes x)=a(x)$, hence $\alpha_{N}b=a$.
This implies that any homomorphism $a:Y\ra N$ such that $Y\in\add(\CR)$ factorizes through $\alpha_{N}$, and so the assertion follows.
\end{proof}

\begin{xrem} It is convenient to note that if $\supp_{\CR}({}_{A}(-,N))=\{X_{1},\dots,X_{m}\}$ and ${}_{A}(X_{i},N)$ is generated, as a $K$-vector space, by the homomorphisms $f_{i}^{1},\dots,f_{i}^{n_{i}}$, for any $i=1,\dots,m$, then $M_{N}\cong\bigoplus_{i=1}^{m}X_{i}^{n_{i}}$ and
$\alpha_{N}\cong\left[f_{1}^{1}\dots f_{1}^{n_{1}}f_{2}^{1}\dots f_{2}^{n_{2}}\dots f_{m}^{1}\dots f_{m}^{n_{m}}\right]$. Although such setting may seem straightforward, the approach presented in the proof above makes argumentation more concise.
\end{xrem}

\section{The main results}

We start with recalling some basic facts on trivial extensions of $K$-algebras and their Galois coverings by some special locally finite dimensional
$K$-algebras (or locally bounded $K$-categories, equivalently). From our point of view, the most important cases of these locally finite dimensional
algebras are \textit{repetitive algebras} \cite{HW} and \textit{cluster repetitive algebras} \cite{ABS}.

Assume that $C$ is an algebra and $E$ is a non-zero $C$-$C$-bimodule. Consider a locally finite
dimensional $K$-algebra $C_{E}$ of the form
$$C_{E}=\left[\begin{array}{ccccc}\ddots&&&&0\\\ddots&C_{-1}&&&\\  &E_{0}&C_{0}&&\\&&E_{1}&C_{1}&\\
 0&& &\ddots&\ddots\end{array}\right]$$ where $C_{i}=C$ and $E_{i}=E$, for any $i\in\ZZ$,
and there are only finitely many non-zero entries. The multiplication is naturally induced from that of $C$ and the $C$-$C$-bimodule structure of
$E$. Further, the identity maps $C_{i}\ra C_{i-1}$ and $E_{i}\ra E_{i-1}$ induce  an automorphism $\nu:=\nu_{C_{E}}$ such that the orbit algebra
$C_{E}\slash\langle\nu\rangle$ is isomorphic to the \textit{trivial extension} $C\ltimes E$ of $C$ by $E$.

Observe that $C_{E}$ may be viewed as a locally bounded $K$-category as follows. Assume that $\{e_{1},\dots,e_{n}\}$ is a complete set of primitive
orthogonal idempotents of $C$. Then the objects of $C_{E}$ are of the form $e_{m,i}$, for $m\in\{1,...,n\}$, $i\in\ZZ$, and the morphism spaces are
defined in the following way
$$C_{E}(e_{m,j},e_{l,i})=\left\{\begin{array}{cl}e_{l}Ce_{m}, & i=j,\\e_{l}Ee_{m}, & i=j+1,\\
0,& \textnormal{otherwise.}\end{array} \right.$$
Moreover, the projection functor $C_{E}\ra C_{E}\slash\langle\nu\rangle\cong C\ltimes E$ is  a
Galois  covering \cite{BoGa} with an  admissible torsion-free covering group $\langle\nu\rangle\cong\ZZ$.

In the case $E=D(C)$, the algebra $C_{E}$ is called the \textit{repetitive algebra} of $C$, denoted by
$\widehat{C}$,  and it is a self-injective algebra \cite{HW}. The automorphism $\nu_{\widehat{C}}$ is the \textit{Nakayama automorphism} of
$\widehat{C}$ and $\widehat{C}\slash\langle\nu_{\widehat{C}}\rangle$ is isomorphic to the trivial extension algebra $T(C)= C\ltimes D(C)$. Let us
mention that the repetitive algebras of tilted algebras play an important role in the classification problems of self-injective algebras, see e.g. \cite{ErKeSk},
\cite{Sk4}.

If $C$ is a tilted algebra and $E=\Ext_{C}^{2}(DC,C)$, then the above matrix algebra $C_{E}$ is called the \textit{cluster repetitive algebra} of $C$ and
denoted  by $\check{C}$ \cite{ABS}. In this case, the trivial extension $C\ltimes E$ of $C$ by $E=\Ext_{C}^{2}(DC,C)$ is called the \textit{relation
extension algebra} and denoted by $\wt{C}$ \cite{ABS2}. It also follows from \cite{ABS2} that $\wt{C}$ is a \textit{cluster-tilted algebra} in the
sense of \cite{BMR} (i.e. the endomorphism algebra of a cluster-tilting object in a cluster category) and every cluster-tilted algebra occurs in that way.

Let $G\colon \check{C}\ra\check{C}\slash\langle\nu_{\check{C}}\rangle\cong\wt{C}$ be the Galois covering of the cluster-tilted algebra $\wt{C}$ by
the cluster repetitive category $\check{C}$. We denote by $G_{\lambda}\colon \mod(\check{C})\ra\mod(\wt{C})$ the associated push-down functor.

\begin{thm} Assume that $C$ is a tilted algebra, $\check{C}$ the associated cluster repetitive
$K$-category and $\widetilde{C}$ the cluster-tilted algebra. There exists a fundamental domain $B$
of $\check{C}$ and hence $\KG(\check{C})=\KG(\widetilde{C})$.
\end{thm}

\begin{proof} Let $C$ be a tilted algebra. Consider the matrix algebra $\overline{C}= \left[\begin{array}{cc}C_0&0\\E&C_1\end{array}\right]$, where
$C_0=C_1=C$ and $E=\Ext_{C}^{2}(DC,C)$, endowed with the ordinary matrix addition and the multiplication induced from that of $C$ and from the
$C$-$C$-bimodule structure of $E=\Ext_{C}^{2}(DC,C)$ (called the \textit{cluster duplicated algebra} of $C$ in \cite{ABS}). Lemma 5 of \cite{ABS}
describes the quiver $Q_{\check{C}}$ of the cluster repetitive category $\check{C}$ and we easily conclude that $\overline{C}$ is a finite convex
subcategory of $\check{C}$. Moreover, in the proof of \cite[Theorem 22]{ABS} there is defined a full subcategory $\check{\Omega}$ of
$\ind(\check{C})$ such that the restriction
$$G_{\lambda}|_{\check{\Omega}}\colon  \check{\Omega} \rightarrow \ind (\widetilde{C})$$ of the push-down functor
$G_{\lambda}\colon \mod(\check{C})\ra\mod(\wt{C})$ is bijective on objects, faithful, preserves irreducible morphisms and almost split
sequences\footnote{We note that $\check{\Omega}$, also denoted by $\Omega$ in \cite{ABS}, is called by the authors a fundamental domain as well. This
definition differs only slightly from ours in Section 2.}. Let us add that the objects of $\check{\Omega}$ are the successors of $\Sigma_0$ and also
proper predecessors of $\Sigma_1$ in $\ind (\check{C})$, where $\Sigma_i$ denotes the image in $\mod(C_i)$ of a complete slice $\Sigma$ from
$\mod(C)$ under the isomorphisms $C_i\cong C$, $i \in \ZZ$. Further, it is shown that $\ob(\overline{C})=\supp (\check{\Omega})$.

Let now $X\in \ind(\check{C})$. Then $G_{\lambda}(X)\in \ind(\widetilde{C})$. Since $G_{\lambda}|_{\check{\Omega}}$ is bijective on objects, there
exists a module $Y \in \check{\Omega}$ such that $G_{\lambda}(X)\cong G_{\lambda}(Y)$ and hence $^gX \cong Y$, for some $g\in G$ (see e.g.
\cite[2.5]{DoSk}). We conclude that for any $X \in \ind(\check{C})$ there is $g \in G$ such that $^gX\in\check{\Omega}$. This implies that
$\supp(^gX) \subseteq \supp(\check{\Omega})=\overline{C}$ which means that $\overline{C}$ is a fundamental domain of $\check{C}$. Applying Theorem
\ref{2.4} to the Galois covering $G\colon \check{C}\ra\check{C}\slash\langle\nu_{\check{C}}\rangle\cong\wt{C}$ we get that
$\KG(\check{C})=\KG(\widetilde{C})$.
\end{proof}

The notation for the subcategory $\check{\Omega}$ ($=\Omega$) in the above theorem comes from
\cite{ABS}. This notation is slightly confusing because of its similarity to the usual symbol for the
syzygy functor, nevertheless we use it to be consistent with \cite{ABS}. From now on $\Omega$ is
reserved for the syzygy functor.

Our aim now is to show that $\KG(\check{C})\leq\KG(\widehat{C})$, for any tilted algebra $C$.
We denote by $\CK_{C}$ the set
$$\{\widehat{P}_{x},\tau^{1-i}\Omega^{-i}(C)\mid x\in(\widehat{C})_{0},i\in\ZZ\}$$ of modules from
$\mod(\widehat{C})$ where $\widehat{P}_{x}$ is an indecomposable projective $\widehat{C}$-module
at the vertex $x\in(\widehat{C})_{0}$ and $\tau=\tau_{\widehat{C}}$
is the Auslander-Reiten translation in $\mod(\widehat{C})$. We have the following fact,
see \cite[Lemma 8, Theorem 9]{ABS}.

\begin{prop}
Assume that $C$ is a tilted algebra. There exists an additive $K$-linear functor $\phi:\mod(\widehat{C})\ra\mod(\check{C})$ which is full and dense
such that $\Ker(\phi)$ equals the class of all homomorphisms in $\mod(\widehat{C})$ which factorize through $\add(\CK_{C})$.
\end{prop}

The following property of $\CK_{C}$ is crucial in applying the results from Section 2.

\begin{prop} Assume that $C$ is a tilted algebra. The class $\CK_{C}$ is a hs-finite class of
$\mod(\widehat{C})$.
\end{prop}

\begin{proof} We view a tilted algebra $C$ as a $\widehat{C}$-module whose support is the set $e_{1,0},\dots,e_{n,0}$ of objects of $\widehat{C}$.
The proof is divided into three cases.

\underline{Case 1.} Let $C$ be a tilted algebra and $N\in\mod(\widehat{C})$. Then ${}_{\widehat{C}}(\widehat{P}_{x},N)\cong N(x)$, for any $x\in
ob(\widehat{C})$.  Hence the number of indecomposable projective modules in $\mod(\widehat{C})$ belonging to the support of the functor
${}_{\widehat{C}}(-,N)$ is finite and equals the cardinality of the set $\supp(N)$. Note that this argument works for an arbitrary finite
dimensional algebra $C$.

\underline{Case 2.} Let $C$ be a tilted algebra of a Dynkin type $\Delta$. Recall that the stable part $\Gamma_{\widehat{C}}^s$ of the
Auslander-Reiten quiver  $\Gamma_{\widehat{C}}$ of $\widehat{C}$ is of the form $\mathbb{Z}\Delta$. Assume that $P$ is an indecomposable direct
summand of $C$. Since $P$ is not a projective-injective $\widehat{C}$-module, it lies on a stable section $\Sigma$ in $\ZZ\Delta$
\cite[VIII.1]{AsSiSk}. The repetitive category $\widehat{C}$ is locally representation finite by \cite{AHR} (see also \cite[Theorem 5.1]{JSk}) and so
we conclude that the support of the functor ${}_{\widehat{C}}(-,N)$ is finite. This yields the existence of some integers $k,m$ such that $k>m$ and
the support of ${}_{\widehat{C}}(-,N)$ is contained in the full subquiver $\CD$ of $\Gamma_{\widehat{C}}$ consisting of all indecomposable modules
lying between $\tau^{k}\Sigma$ and $\tau^{m}\Sigma$ (more formally, all modules being both successors of $\tau^{k}\Sigma$ and predecessors of
$\tau^{m}\Sigma$).

In what follows we assume that the module $P$ is a successor of the module $N$ and hence $k>m\geq 0$. For other cases the arguments are similar.

Assume now that $X$ is an arbitrary indecomposable non-projective (and non-injective) $\widehat{C}$-module. Then $\Omega(X)$ is a proper predecessor
of  $X$ and $\Omega^{-1}(X)$ is a proper successor of $X$. Indeed, this follows from the fact that there are non-zero homomorphisms $\Omega(X)\ra
P(X)$, $P(X)\ra X$, $X\ra I(X)$ and $I(X)\ra\Omega^{-1}(X)$ where $P(X)$ and $I(X)$ denote the projective cover of $X$ and the injective envelope of
$X$ in $\mod(\widehat{C})$, respectively (recall that there are no oriented cycles in $\Gamma_{\widehat{C}}$). This implies that for $i>0$ we have
$\tau^{1-i}\Omega^{-i}(P)\in\tau^{l}\Sigma$ for some $l<0$ and thus $\tau^{1-i}\Omega^{-i}(P)\notin\CD$. Moreover, if $1-i>k$ we get
$\tau^{1-i}\Omega^{-i}(P)\in\tau^{s}\Sigma$ for some $s>k$ and hence $\tau^{1-i}\Omega^{-i}(P)\notin\CD$ also in this case. Therefore we conclude
that if $\tau^{1-i}\Omega^{-i}(P)\in\CD$, then $i\in\{-k+1,-k+2,\dots,-1,0\}$. Note that any direct summand of a module $\tau^{1-i}\Omega^{-i}(C)$ is
of the form $\tau^{1-i}\Omega^{-i}(P)$, for some direct summand $P$ of $C$. Together with Case 1, these arguments yield the class $\CK_{C}$ is
hs-finite in the case $C$ is tilted of Dynkin type.

\underline{Case 3.} Let $C$ be a tilted algebra of Euclidean or wild type $\Delta$. Then the Auslander-Reiten quiver $\Gamma_{\widehat{C}}$ of
$\widehat{C}$  has a decomposition $$\Gamma_{\widehat{C}}=\bigvee_{q\in\mathbb{Z}}(\mathcal{X}_q\vee \mathcal{C}_q)$$ such that for each $q \in
\mathbb{Z}$, $\mathcal{X}_q$ is a component whose stable part $\mathcal{X}_q^s$ is of the form $\mathbb{Z}\Delta$ and $\mathcal{C}_q$ is an infinite
family of components whose stable part $\mathcal{C}_q^s$ is a union either of stable tubes (if $\Delta$ is of Euclidean type \cite{ANSk}, \cite{Sk4})
or of components of the form $\mathbb{ZA}_{\infty}$ (if $\Delta$ is of wild type \cite{ErKeSk}*{3.5}). Moreover, the following statements hold:
\begin{enumerate}[\rm (a)]
\item for each  pair $p,\,q\in\mathbb{Z}$ with $q>p$,
 we have $\mbox{Hom}\,_{\widehat{C}}(\mathcal{X}_q,\mathcal{X}_p\vee \mathcal{C}_p)=0$
 and $\mbox{Hom}\,_{\widehat{C}}(\mathcal{C}_q, \mathcal{C}_p\vee\mathcal{X}_{p+1})=0$;
\item for each  $q\in\mathbb{Z}$, we have
 $\nu_{\widehat{C}}(\mathcal{X}_q)=\mathcal{X}_{q+2}$ and
 $\nu_{\widehat{C}}(\mathcal{C}_q)=\mathcal{C}_{q+2}$;
\item for each  $q\in\mathbb{Z}$, we have $\mbox{Hom}\,_{\widehat{C}}(\mathcal{X}_q,\mathcal{X}_p \vee \mathcal{C}_p)=0$,
 $\mbox{Hom}\,_{\widehat{C}}(\mathcal{C}_q,\mathcal{C}_p \vee \mathcal{X}_{p+1})=0$ for $p>q+2$;
\item for each  $q\in\mathbb{Z}$, we have $\Omega_{\widehat{C}}(\mathcal{C}^s_{q+1})=\mathcal{C}^s_q$
 and $\Omega_{\widehat{C}}(\mathcal{X}^s_{q+1})=\mathcal{X}^s_q$.
\end{enumerate}
Let $N \in\mathcal{X}_t\vee \mathcal{C}_t$ for some $t\in \ZZ$. By (a) we have that the support of $_{\widehat{C}}(-, N)$ consists of modules from
$\bigvee_{p\leq t}(\mathcal{X}_p\vee \mathcal{C}_p)$. Hence by (c) we obtain that the support of $_{\widehat{C}}(-, N)$ is contained in the family
$$\mathcal{D}=(\mathcal{X}_{t-3}\vee \mathcal{C}_{t-3})\vee(\mathcal{X}_{t-2}\vee \mathcal{C}_{t-2}) \vee (\mathcal{X}_{t-1}\vee \mathcal{C}_{t-1})\vee
  (\mathcal{X}_{t}\vee \mathcal{C}_{t}).$$
Assume that $P$ is an indecomposable direct summand of $C$. Since $P$ is not a projective-injective $\widehat{C}$-module, we obtain that $P\in
\mathcal{X}_{0}^{s}\vee \mathcal{C}_{0}^{s}\vee\mathcal{X}_{1}^{s}\vee\mathcal{C}_{1}^{s}$, see the description of supports of indecomposable
$\widehat{C}$-modules in \cite{ANSk,HW} and \cite{ErKeSk}*{3.5}. Fix some $i\in\ZZ$. Then it follows by (d) that
$\tau^{1-i}\Omega^{-i}(P)\in\mathcal{X}_{i}^{s}\vee \mathcal{C}_{i}^{s}\vee\mathcal{X}_{i+1}^{s}\vee \mathcal{C}_{i+1}^{s}$. Therefore we conclude
that if $\tau^{1-i}\Omega^{-i}(P)\in\CD$, then at least one of the following conditions holds: $\mathcal{X}_{i}^{s}\subseteq\CD$,
$\mathcal{C}_{i}^{s}\subseteq\CD$, $\mathcal{X}_{i+1}^{s}\subseteq\CD$ or $\mathcal{C}_{i+1}^{s}\subseteq\CD$. Equivalently, if
$\tau^{1-i}\Omega^{-i}(P)\in\CD$, then $i\in\{t-4,t-3,\dots,t\}$. As before, any direct summand of a module $\tau^{1-i}\Omega^{-i}(C)$ is of the form
$\tau^{1-i}\Omega^{-i}(P)$, for some direct summand $P$ of $C$. Hence, together with Case 1, these arguments yield the class $\CK_{C}$ is hs-finite
in the case $C$ is tilted of Euclidean or wild type.

Summing up, the above three cases show that the class $\CK_{C}$ is a hs-finite class of $\mod(\widehat{C})$ for any tilted algebra $C$.
\end{proof}

We are now in a position to prove the first main result of the paper.

\begin{thm} Assume that $C$ is a tilted algebra and $\widehat{C},\check{C},\wt{C}$ are the associated repetitive category, cluster repetitive category
            and cluster-tilted algebra, respectively. The following assertions hold.
\begin{enumerate}[\rm(1)]
	\item The functor $\phi\colon \mod(\widehat{C})\ra\mod(\check{C})$ is admissible.
	\item We have $\KG(\wt{C})=\KG(\check{C})\leq\KG(\widehat{C})$.
\end{enumerate}
\end{thm}

\begin{proof} $(1)$ Proposition 3.3 yields the class $\CK_{C}$ is hs-finite, so the functor $\phi\colon$ $\mod(\widehat{C})\ra\mod(\check{C})$ is admissible by
Proposition 3.2, Theorem \ref{2.6} and Lemma \ref{2.7}. The assertion of $(2)$ follows from $(1)$, Theorem \ref{2.6} and Theorem 3.1.
\end{proof}

We shall apply the following classification of Krull-Gabriel dimensions of locally support-finite repetitive $K$-categories over an algebraically
closed  field $K$ \cite{P4}.

\begin{thm} Assume that $K$ is an algebraically closed field and $A$ is a finite dimensional basic and connected $K$-algebra such that $\widehat{A}$ is
            locally support-finite. Then $\KG(\widehat{A})\in\{0,2,\infty\}$ and the following assertions hold.
\begin{enumerate}[\rm(1)]
	\item $\KG(\widehat{A})=0$ if and only if $\widehat{A}\cong\widehat{B}$ where $B$ is some tilted algebra of Dynkin type.
	\item $\KG(\widehat{A})=2$ if and only if $\widehat{A}\cong\widehat{B}$ where $B$ is some representation-infinite tilted algebra of Euclidean type.
	\item $\KG(\widehat{A})=\infty$ if and only if $\widehat{A}$ is wild or $\widehat{A}\cong\widehat{B}$ where $B$ is a tubular algebra. \epv
\end{enumerate}
\end{thm}

The following theorem is the second main result of the paper. In this theorem we refine the assertion $(2)$ of Theorem 3.4 and determine the
Krull-Gabriel  dimension of cluster-tilted algebras. We also conclude that the class of cluster-tilted algebras supports the conjecture of Prest.

\begin{thm} \label{3.6}
             Assume that $K$ is an algebraically closed field, $C$ is a tilted $K$-algebra and $\widehat{C},\check{C},\wt{C}$ are the associated repetitive
           category, cluster repetitive category and cluster-tilted algebra, respectively. Then $\KG(\wt{C})=\KG(\check{C})=\KG(\widehat{C})\in\{0,2,\infty\}$
            and the following assertions hold.
\begin{enumerate}[\rm(1)]
	\item $C$ is tilted of Dynkin type if and only if $\KG(\wt{C})=0$.
	\item $C$ is tilted of Euclidean type if and only if $\KG(\wt{C})=2$.
	\item $C$ is tilted of wild type if and only if $\KG(\wt{C})=\infty$.
\end{enumerate} In particular, a cluster-tilted algebra $\wt{C}$ has finite Krull-Gabriel dimension if and only if $\wt{C}$ is of domestic representation type.
\end{thm}

\begin{proof} We apply freely Theorem 3.5 and the assertion $(2)$ of Theorem 3.4 stating that $\KG(\wt{C})=\KG(\check{C})\leq\KG(\widehat{C})$. We prove  all
assertions simultaneously.

If $C$ is of Dynkin type, then $\KG(\wt{C})=\KG(\check{C})\leq\KG(\widehat{C})=0$, so
$\KG(\wt{C})=\KG(\check{C})=\KG(\widehat{C})=0$.\footnote{Observe that this argument gives another proof of the fact that cluster-tilted algebras of
Dynkin type are representation finite.}

If $C$ is of Euclidean type, then $\KG(\wt{C})=\KG(\check{C})\leq\KG(\widehat{C})=2$, but $\KG(\wt{C})\neq 0$ since $\wt{C}$ is of infinite
representation  type (see \cite{BMR}) and $\KG(\wt{C})\neq 1$ by \cite{Kr2}. This yields $\KG(\wt{C})=\KG(\check{C})=\KG(\widehat{C})=2$.

If $C$ is of wild type, then $\wt{C}$ is also of wild type by \cite{BMR} and hence we obtain $\KG(\wt{C})=\KG(\check{C})=\KG(\widehat{C})=\infty$.

Since the algebra $C$ is a tilted algebra either of Dynkin, or of Euclidean or of wild type, we conclude that the above implications can be replaced by
equivalences. This also yields the fact that $\KG(\wt{C})=\KG(\check{C})=\KG(\widehat{C})\in\{0,2,\infty\}$. Moreover, if $C$ is tilted of Dynkin or
Euclidean type, then $\wt{C}$ is of domestic representation type (assuming that finite type is contained in domestic type). This implies that the
class of cluster-tilted algebras supports the conjecture of Prest on Krull-Gabriel dimension and domestic algebras.
\end{proof}

We finish the section with natural application of the above theorem to the theory of \textit{super-decomposable pure-injective modules}. Assume that
$R$  is a ring with a unit. An $R$-module $M\neq 0$ is \textit{super-decomposable} if and only if $M$ does not have an indecomposable direct summand.
For the concept of \textit{pure-injectivity} we refer to \cite{JeLe}. The problem of the existence of super-decomposable pure-injective $R$-modules
is studied for the first time in \cite{Zi}. The case when $R$ is a finite dimensional algebra over a field is studied in many papers, see
\cite{KaPa3} for an up-to-date list of results concerning this case and \cite[Theorem 8.3]{P4} for the most recent one about self-injective algebras.
It is conjectured by Prest that if $R$ is a finite dimensional algebra over an algebraically closed field, then $R$ is of domestic representation
type if and only if there is no super-decomposable pure-injective $R$-module, see for example \cite{Pr2}. This conjecture is sometimes restricted
only to countable fields, see \cite{KaPa3} for details. The following theorem supports the conjecture.

\begin{cor} Assume that $C$ is a tilted algebra over an algebraically closed field $K$ and $\wt{C}$ is the corresponding cluster-tilted algebra. If $\wt{C}$ is of domestic type, then it has no super-decomposable pure-injective module. The converse implication holds if the field $K$ is countable.
\end{cor}

\begin{proof} If $\wt{C}$ is domestic, then by assertions $(1),(2)$ of Theorem 3.6 we have that $\KG(\wt{C})$ is finite and thus super-decomposable pure-injective $\wt{C}$-module does not exist, see for example \cite{Pr2}. Conversely, if $\wt{C}$ is non-domestic, then $\wt{C}$ is wild, so it possesses a super-decomposable pure-injective module by \cite[Theorem 3.2]{P5}.
\end{proof}

\appendix
\section{Appendix (by Grzegorz Bobi{\'n}ski)}

Throughout this appendix $K$ is a fixed field. All considered algebras are finite dimensional $K$-algebras. For simplicity we also assume that all
considered categories are Krull--Schmidt $K$-categories with finite dimensional homomorphism spaces.

The aim of this appendix is to present an alternative proof of the following result proved in Theorem~\ref{3.6}.

\begin{thm}
Let $C$ be a cluster-tilted algebra.
\begin{enumerate}
\renewcommand{\labelenumi}{\textup{(\arabic{enumi})}}

\item If $C$ is of Dynkin type, then $\KG (C) = 0$.

\item If $C$ is of Eulidean type, then $\KG (C) = 2$.

\item If $C$ is of wild type, then $\KG (C) = \infty$.

\end{enumerate}
\end{thm}

In fact we prove the following equivalent version of the above result.

\begin{thm} \label{theo mainbis}
If $H$ is a hereditary algebra and $C$ is a cluster-tilted algebra of type $H$, then $\KG (C) = \KG (H)$.
\end{thm}

The above mentioned equivalence follows from the following well-known description of the Krull--Gabriel dimension of the hereditary algebras.

\begin{prop}
Let $H$ be a hereditary algebra.
\begin{enumerate}
\renewcommand{\labelenumi}{\textup{(\arabic{enumi})}}

\item If $H$ is of Dynkin type, then $\KG (H) = 0$.

\item If $H$ is of Eulidean type, then $\KG (H) = 2$.

\item If $H$ is of wild type, then $\KG (H) = \infty$.

\end{enumerate}
\end{prop}

\begin{proof}
(1) follows from~\cite[Corollary~3.14]{Au}, (2) from~\cite[Theorem~4.3]{Geigle1985}, and (3) from~\cite[Thereom~4.3]{Baer}.
\end{proof}

We recall first the definition of the Krull--Gabriel dimension of an abelian category. Let $\CA$ an abelian category and $\CA_{-1} = 0$. For any
$\alpha$ being either an ordinal number or $-1$, let $\CA_{\alpha + 1}$ be the Serre subcategory of $\CA$ consisting of those objects in $\CA$ which
have finite length after passing to the quotient category $\CA / \CA_\alpha$. Moreover, if $\beta$ is a limit ordinal, then $\CA_\beta =
\bigcup_{\alpha < \beta} \CA_\alpha$. By the Krull--Gabriel dimension $\KGdim (\CA)$ of $\CA$ we mean the smallest ordinal number $\alpha$ such that
$\CA_\alpha = \CA$. If there is not such number, then we put $\KGdim (\CA) = \infty$.

Abelian categories we are interested in are of special form. Namely, let $\CC$ be an additive category and denote by $\CF (\CC)$ the category of all
contravariant finitely presented functors from $\CC$ to the category $\mod K$ of finite dimensional vector spaces. A category $\CC$ is called
coherent, if the category $\CF (\CC)$ is abelian. There are two important examples of coherent categories: abelian and triangulated ones. If $C$ is
an algebra, then we put $\CF (C) = \CF (\mod C)$ and $\KG (C) = \KGdim (\CF (C))$, where $\mod C$ is the category of finite dimensional $C$-modules.

If $\CB$ is a full subcategory of a category $\CC$, then we denote by $[\CB]$ the ideal of morphisms in $\CC$, which factor through objects in $\CB$.
Next, if $X$ is an indecomposable object of $\CC$, then we denote by $S_X$ the functor $S_X \colon \CC \to \mod K$ such that $S_X (X) = K$ and $S_X
(Y) = 0$, for each indecomposable object $Y$ of $\CC$ nonisomorphic to $X$. The following result due to Geigle~\cite[Corollary~2.9]{Ge2} will play a
crucial role in our proof.

\begin{prop} \label{prop Geigle}
Let $\CC$ be a coherent category and $\CB$ be a full subcategory of $\CC$ with only finitely many indecomposable objects up to isomorphism. If $S_X
\in \CF (\CC)$, for each indecomposable object $X$ of $\CB$, then
\[
\KGdim (\CF (\CC)) = \KGdim (\CF (\CC / [\CB])). \eqno \qed
\]
\end{prop}

Observe that $X \in \CF (\CC)$, for an indecomposable object $X$ of $\CC$, if and only if there is a source map for $X$, i.e.\ a map $f \colon X \to
M$ such that $f$ is not a section and every $f' \colon X \to M'$ which is not a section factors through $f$.

Let $H$ be a hereditary algebra. One defines the cluster category $\CC_H$ as the quotient of the bounded derived category $\CD^b (\mod H)$ by the
action of the functor $\tau^{-1} \circ \Sigma$, where $\tau$ is the Auslander--Reiten translation and $\Sigma$ is the shift functor. It is known that
$\CC_H$ is a triangulated category, whose shift functor is induced by the shift functor $\Sigma$ in $\CD^b (\mod H)$~\cite[Theorem~1]{Ke1}, with
almost split triangles~\cite[Proposition~1.3]{BMRRT}. In particular, there is a source map for every indecomposable object of $\CC_H$, hence we may
apply Proposition~\ref{prop Geigle} with $\CC = \CC_H$.

An object $T$ of $\CC_H$ is called cluster-tilting if $\Hom_{\CC_H} (T, \Sigma T) = 0$ and if for each indecomposable object $X$ of $\CC_H$ the
equality $\Hom_{\CC_H} (X, \Sigma T) = 0$ implies that $X$ is a direct summand of $T$. It is known that every cluster-tilting objects has $n$
pairwise nonisomorphic indecomposable direct summands, where $n$ is the number of pairwise nonisomorphic simple
$H$-modules~\cite[Theorem~3.3]{BMRRT}. The cluster-tilted algebras of type $H$ are by definition the opposite algebras of the endomorphism algebras
of the cluster-tilting objects in $\CC_H$. The following fundamental result~\cite[Theorem~A]{BMR} will be of importance to us.

\begin{prop} \label{prop BMR}
Let $H$ be a hereditary algebra, $T$ a cluster-tilting object, and $C = \End_{\CC_H} (T)^{\op}$. Then $\Hom_{\CC_H} (T, -)$ induces an equivalence
\[
\CC_H / [\add \Sigma T] \simeq \mod C. \eqno \qed
\]
\end{prop}

As an immediate consequence of Propositions~\ref{prop Geigle} and~\ref{prop BMR} we obtain the following.

\begin{cor} \label{coro}
If $H$ is hereditary algebra and $C$ a cluster-tilted algebra of type $H$, then
\[
\KG (C) = \KGdim (\CF (\CC_H)).
\]
\end{cor}

\begin{proof}
Let $C = \End_{\CC_H} (T)^{\op}$, for a cluster-tilting object $T$ in $\CC_H$. We know from Propsition~\ref{prop BMR} that
\[\KG (C) = \KGdim (\CF (\CC_H / [\add \Sigma T])),
\]
thus it is sufficient to apply Proposition~\ref{prop Geigle} with $\CC = \CC_H$ and $\CB = \add \Sigma T$.
\end{proof}

Now we are ready to prove Theorem~\ref{theo mainbis}.

\begin{proof}[Proof of Theorem~\ref{theo mainbis}]
Let $C$ be a cluster-titled algebra of type $H$. Then
\begin{equation} \label{eq one}
\KG (C) = \KGdim (\CF (\CC_H)),
\end{equation}
by Corollary~\ref{coro}. On the other hand, it is well known that $H$ itself is a cluster-tilted algebra of type $H$ ($H$ viewed as an object of
$\CC_H$ is a cluster-tilting object with $\End_{\CC_H} (H)^{\op} \simeq H$), hence using again Corollary~\ref{coro} we obtain
\begin{equation} \label{eq two}
\KG (H) = \KGdim (\CF (\CC_H)).
\end{equation}
Now the claim follows from~\eqref{eq one} and~\eqref{eq two}.
\end{proof}

\paragraph{Acknowledgments}
This research has been supported by the National Science Centre grant no.~2020/37/B/ST1/00127.

\bibsection

\end{document}